\documentclass[letterpaper, 10pt, conference]{ieeeconf}
\usepackage[utf8]{inputenc} 
\DeclareUnicodeCharacter{00A0}{ }
\usepackage{graphicx}%
\usepackage{amsmath}%
\usepackage{amsfonts}%
\usepackage{amssymb}%
\usepackage{framed}
\IEEEoverridecommandlockouts
\def\boxit#1{\vbox{\hrule\hbox{\vrule\kern3pt
        \vbox{\kern3pt#1\kern3pt}\kern3pt\vrule}\hrule}}

\def\reals{ { {\rm  I \kern-0.15em R }  } }
\def\complex{ {\,{{\rm C} \kern-0.50em \raise0.20ex {  |}}\, }}

\def\Rbf{{\bf R}}

\def\Rxx{\Rbf_{\ssstyle X\kern-.1em X}}

\let\ssstyle=\scriptscriptstyle

\def\Kout{\setbox1=\hbox{\Huge\bf K}\hbox to
1.05\wd1{\hspace{.05\wd1}
}}
\def\Sout{\setbox1=\hbox{\Huge\bf S}\hbox to 1.05\wd1{\hspace{.05\wd1}
}}

\newtheorem{theorem}{Theorem}

\newtheorem{lemma}{Lemma}

\title{\LARGE \bf Online Algorithms for the Multi-Armed Bandit Problem with Markovian Rewards}
\author{Cem Tekin, Mingyan Liu
\thanks{C. Tekin and M. Liu are with the Dept. of Electrical Engineering and Computer Science, University of Michigan, Ann Arbor, MI 48105, \{cmtkn, mingyan\}@eecs.umich.edu.}
}

\begin{document}
\maketitle
\thispagestyle{empty}
\pagestyle{empty}

\begin{abstract}
We consider the classical multi-armed bandit problem with Markovian rewards. When played an arm changes its state in a Markovian fashion while it remains frozen when not played. The player receives a state-dependent reward each time it plays an arm. The number of states and the state transition probabilities of an arm are unknown to the player.  
The player's objective is to maximize its long-term total reward by learning the best arm over time. We show that under certain conditions on the state transition probabilities of the arms, a sample mean based index policy achieves logarithmic regret uniformly over the total number of trials. The result shows that sample mean based index policies can be applied to learning problems under the rested Markovian bandit model without loss of optimality in the order. Moreover, comparision between Anantharam's index policy and UCB shows that by choosing a small exploration parameter UCB can have a smaller regret than Anantharam's index policy.
\end{abstract}

\bibliographystyle{ieeetr}

\section{Introduction} \label{sec:intro}

In this paper we study the single player multi-armed bandit problem where the reward of each arm is generated by a Markov chain with unknown statistics, and the states of the Markov chain evolves only when the arm is played. We will investigate the performance of an index policy that depends only on the sample mean reward of an arm. 

In the classical multi-armed bandit problem, originally proposed by Robbins \cite{robbins2}, a gambler (or player) must decide which one of the $K$ machines (or arms) to activate (or play) at each discrete step in a sequence of trials so as to maximize his long term reward.  Every time he plays an arm, he receives a reward (or payoff).  The structure of the reward for each arm is unknown to the player {\em a priori}, but in most prior work the reward has been assumed to be independently drawn from a fixed (but unknown) distribution.  The reward distribution in general differs from one arm to another, therefore the player must use all his past actions and observations to essentially ``learn'' the quality of these arms (in terms of their expected reward) so he can keep playing the best arm.  

This problem is a typical example of the trade-off between {\em exploration} and {\em exploitation}.  On the one hand, the player needs to sufficiently explore or sample all arms so as to discover with accuracy the best arm and avoid getting stuck playing an inferior one erroneously believed to be the best.  On the other hand, the player needs to avoid spending too much time sampling the arms and collecting statistics and not playing the best arm often enough to get a high return.  

Within this context, the player's performance is typically measured by the notion of {\em regret}. It is defined as the difference between the expected reward that can be gained by an ``infeasible'' or ideal policy, i.e., a policy that requires either a priori knowledge of some or all statistics of the arms or hindsight information, and the expected reward of the player's policy.  The most commonly used infeasible policy is the {\em best single action} policy, that is optimal among all policies that continue to play the same arm.  An ideal policy could play for instance the arm that has the highest expected reward (which requires statistical information but not hindsight).  This type of regret is sometimes also referred to as the {\em weak regret}, see e.g., work by Auer et al. \cite{auer2}.  In this study we will only focus on this definition of regret. 

Most studies in this area assume iid rewards, with the notable exception of \cite{anantharam1}.  In \cite{lai1} rewards are modeled as single-parameter univariate densities. Under some conditions such as the denseness of the parameter space and continuity of the Kullback-Leibler number between two densities, Lai and Robbins \cite{lai1} give a lower bound on the regret and construct policies that achieve this lower bound which are called {\em asymptotically efficient} policies. This result is extended by Anantharam et al. in \cite{anantharam2} to the case where playing more than one arm at a time is allowed. Using a similar approach Anantharam et al. in \cite{anantharam1} develops index policies that are asymptotically efficient for arms with rewards driven by finite, irreducible, aperiodic Markov chains with identical state spaces and single-parameter families of stochastic transition matrices.  Agrawal in \cite{agrawal1} considers sample mean based index policies for the iid model that achieve $O(\log n)$ regret, where $n$ is the total number of plays, with a constant that depends on the Kullback-Leibler number. He imposes conditions on the index functions for which they can be treated as upper confidence bounds and generates these index functions for specific one-parameter family of distributions.   Auer et al. in \cite{auer} also proposes sample mean based index policies for iid rewards with bounded support; these are derived from \cite{agrawal1}, but are simpler than the those in \cite{agrawal1} and are not restricted to a specific family of distributions. These policies achieve logarithmic regret uniformly over time rather than asymptotically in time, but have bigger constant than that in \cite{lai1}.   \cite{auer} also proposes randomized policies that achieve logarithmic regret uniformly over time by using an {\em exploration factor} that is inversely proportional to time.

Other works such as \cite{liu2, anandkumar} consider the iid multiarmed bandit problem in the multiuser setting. Players selecting the same arms experience collision according to a certain collision model. In \cite{anandkumar} when a collision occurs on an arm, none of the players selecting that arm receive a reward. In \cite{liu2} an additional collision model is considered where one of the colliding players gets the reward.  Assume that there are $M$ players and $M$ is less than the number of arms $K$. The main idea underlying the policies used in such multi-user setting is to encourage the players to play the best $M$ arms, while playing the other arms only logarithmically.

Its worth noting that when the reward process is iid, the question of ``what happens to the arms that are not played'' does not arise.   This is because whether the unselected arms remain still (frozen in their current states) or transition to another state with a different reward is inconsequential; in either case the player does not obtain the reward from arms he does not play.  Since the rewards are independently drawn each time, remaining still or not does not affect the reward the arm produces the next time it is played.  This simplifies the problem significantly if the physical system represented by such a multiarmed bandit model is such that the arms cannot stay still (or are {\em restless}).  
This unfortunately is not the case with Markovian rewards.  There is a clear difference between whether the arms are rested or restless.  In the rested case, since the state is frozen when an arm is not played, 
the state we next observe the arm to be in is {\em independent} of how much time elapses before we play the arm again.   In the restless case, the state of an arm continues to evolve accordingly to the underlying Markov law regardless of the player's decision, but the actual state is not observed nor the reward obtained unless the arm is chosen by the player.  Clearly in this case 
the state we next observe it to be in is now {\em dependent} on the amount of time that elapses between two plays of the same arm.  This makes the problem significantly more difficult.  To the best of our knowledge, there has been no study of the restless bandits in this learning context.  In this paper we will only focus on the rested case. 

As \cite{anantharam1} is the most closely related to the present paper, we now elaborate on the differences between the two. 
In \cite{anantharam1} the rewards are assumed to be generated by rested Markov chains with transition probability matrices parametrized by a single index $\theta$.  This implies that for states $x,y$ in the state space of arm $i$ the transition probability from state $x$ to $y$ is given by $p_{xy}(\theta)$, where the player knows the function $p_{xy}(\theta)$ but not the value of $\theta$.  Because of this assumption the problem more or less reduces to a single-parameter estimation problem. Indices are formed in \cite{anantharam1}  using natural statistics to test the hypothesis that the rewards from an arm are generated by a parameter value less than $\theta$ or by $\theta$.  The method also requires log-concavity of $p$ in $\theta$ as an assumption 
in order to have a test statistic increasing in $\theta$.   
By contrast, in the present paper we do not assume the existence of such a single-parameter function $p$, or if it does exist it is unknown to the player.  We however do require that the Markovian reward process is reversible. Secondly, while \cite{anantharam1} assumes that arms have identical state spaces, our setting allows arms to have different state spaces.  Thirdly, there is no known recursive methods to compute the indices used in \cite{anantharam1} which makes the calculation hard, while our indices depend on the sample mean of the arms which can be computed recursively and efficiently. 
Finally, the bound produced in \cite{anantharam1} holds asymptotically in time, while ours (also logarithmic) holds uniformly over time. 
We do, however, use very useful results from \cite{anantharam1} in our analysis as discussed in more detail in subsequent sections. 
It should be noted that our results is not a generalization of that in \cite{anantharam1} since the two regret bounds, while of the same order, have different constants. 

Our main results are summarized as follows. 
\begin{enumerate}
\item We show that when each arm is given by a finite state irreducible, aperiodic, and reversible\footnote{Note that reversibility is actually not necessary for our main result, i.e., the logarithmic regret bound, to hold, provided that we use a large deviation bound from \cite{lezaud} rather than from \cite{gillman} as we have done in the present paper.} Markov chain with positive rewards, and under mild assumptions on the state transition probabilities of the arms, there exist simple sample mean based index policies that achieve logarithmic regret uniformly over time.  
\item We interpret the conditions on the state transition probabilities in a simple model where arms are modeled as two-state Markov chains with identical rewards.
\item We compare numerically the regret of our sample mean based index policy under different values of the exploration parameter. We also compare our policy with the index policy given in \cite{anantharam1}. 
\end{enumerate}

We end this introduction by pointing out another important class of multi-armed bandit problems solved by Gittins \cite{gittins1}.  The problem there is very different from the one considered in this study (it was referred to as the {\em deterministic} bandit problem by \cite{lai1}), in that the setting of \cite{gittins1} is such that the rewards are given by Markov chains whose statistics are perfectly known a priori. Therefore the problem is one of {\em optimization} rather than {\em exploration} and {\em exploitation}: the goal is to determine offline an optimal policy of playing the arms so as to maximize the total expected reward over a finite or infinite horizon. 
  
The remainder of this paper is organized as follows. In Section \ref{sec:problem} we formulate the single player rested Markovian bandit problem and relate the expected regret with expected number of plays from arms. In Section \ref{sec:policy} we propose an index policy based on \cite{auer} and analyze the regret of that policy in the rested Markovian model.  Discussion on this result is given in Section \ref{sec:discussion}. 
In Section \ref{sec:app} we give an application that can be modeled as a rested bandit problem and evaluate the performance of our policy for this application. Finally, Section \ref{sec:conc} concludes the paper.


\section{Problem Formulation and Preliminaries} \label{sec:problem}

We assume that there are $K$ arms indexed by $i = 1,2, \cdots, K$.  The $i$th arm is modeled as an irreducible Markov chain with finite state space $S^i$. Rewards drawn from a state of an arm is stationary and positive. Let $r^i_x$ denote the reward obtained from state $x$ of arm $i$. Let $P^i=\left\{p_{xy}^i, x,y \in S^i \right\}$ denote the transition probability matrix of arm $i$. We assume the arms (i.e., the Markov chains) are mutually independent.  The  mean reward from arm $i$, denoted by $\mu^i$, is the expected reward of arm $i$ under its stationary distribution $\boldsymbol{\pi}^i = (\pi^i_x, x\in S^i)$.  Then, 
\begin{eqnarray}
\mu^i=\displaystyle\sum_{x \in S^i} r^i_x \pi^i_x ~. 
\end{eqnarray}
For convenience, we will use $^*$ in the superscript to denote the arm with the highest mean.  For instance, $\mu^*=\max_{1\leq i\leq K} \mu^i$.  
For a policy $\alpha$ we define its regret $R^\alpha(n)$ as the difference between the expected total reward that can be obtained by playing the arm with the highest mean and the expected total reward obtained from using policy up to time $n$.  Let $\alpha(t)$ be the arm selected by the policy $\alpha$ at $t$, and $x_{\alpha(t)}$ the state of arm $\alpha(t)$ at time $t$.  Then we have
\begin{eqnarray}
R^\alpha(n)=n\mu^*-E^\alpha\left[\displaystyle\sum_{t=1}^n r^{\alpha(t)} _{x_{\alpha(t)}} \right] ~. 
\end{eqnarray}
The objective of the study is to examine how the regret $R^\alpha(n)$ behaves as a function of $n$ for a given policy $\alpha$, through appropriate bounding. 

Note that playing the arm with the highest mean is the optimal policy among all {\em single-action} policies.  It is, however, not in general the optimal policy among all stationary and nonstationary policies if all statistics are known a priori.  The optimal policy in this case (over an infinite horizon) is the Gittins index policy, first given by Gittins in his seminal paper \cite{gittins1}.  In the special case where rewards of each arm is iid, the optimal policy over all stationary and nonstationary policies is indeed the best single-action policy.   In this study we will restrict our performance comparison to the best single-action policy. 

To proceed, below we introduce a number of preliminaries that will be used in later analysis.  The key to bounding $R^\pi(n)$ is to bound the expected number of plays of any suboptimal arm.  
Let $T^{\alpha, i}(t)$ be the total number of times arm $i$ is selected by policy $\alpha$ up to time $t$.   
We first need to relate regret $R^\alpha(n)$ with $E^\alpha\left[T^{\alpha, i}(n)\right]$. 
We use the following lemma,  which is Lemma 2.1 from \cite{anantharam1}.  The proof is reproduced here for completeness. 

\begin{lemma}\label{lemma1}
[Lemma 2.1 from \cite{anantharam1}] Let $Y$ be an irreducible aperiodic Markov chain with a state space $S$, transition probability matrix $P$,  an initial distribution that is non-zero in all states, and 
a stationary distribution $\{\pi_x\}, \forall x\in S$. Let $F_t$ be the $\sigma$-field generated by random variables $X_1,X_2,...,X_t$ where $X_t$ corresponds to the state of the chain at time $t$.  Let $G$ be a $\sigma$-field independent of $F=\vee_{t \geq 1} F_t$, the smallest $\sigma$-field containing $F_1, F_2, ...$. Let $\tau$ be a stopping time with respect to the increasing family of $\sigma$-fields $\left\{G \vee F_t, t \geq 1\right\}$. Define  $N(x,\tau)$ such that
\begin{eqnarray*}
N(x,\tau)=\displaystyle\sum_{t=1}^\tau I(X_t = x).  
\end{eqnarray*}
Then $\forall \tau$ such that $E\left[\tau\right]<\infty$, we have 
\begin{eqnarray}
\left|E\left[N(x,\tau)\right]-\pi_xE\left[\tau\right]\right|\leq C_{P}, 
\end{eqnarray}
where $C_{P}$ is a constant that depends on $P$.
\end{lemma}

\begin{proof} 
Define $\left\{F_t, t\geq 1\right\}$ stopping times
\begin{eqnarray*}
\tau_k &=& \inf \left\{t > \tau_{k-1} | X_t=X_1\right\}, k=1,2,...\\
\tau_0 &=& 1.
\end{eqnarray*}

Because of irreducibility, $\tau_k < \infty$. Define $B_k$ as the $k$th block. For a sample path $w$ it is the sequence $(x_{\tau_{k-1} (w)},  x_{\tau_{k-1} (w)+1}, \cdots, x_{\tau_{k} (w)-1})$.
Then,
\begin{eqnarray*}
F_{\tau_k}=\sigma(B_1,...,B_k).
\end{eqnarray*}

Let $S^*=\cup_{t \geq 1} S^t$ be the Borel $\sigma$-field of the discrete topology. For $x,y \in S, \mathbf{y}=(y_1,y_2,...,y_t) \in S^*$, let $l(\mathbf{y})$ be the length of $\mathbf{y}$. Then by the regenerative cycle theorem sequence $\left\{B_k\right\}$ is i.i.d and
\begin{eqnarray*}
EN(x,B_1)=\pi_x El(B_1)
\end{eqnarray*}

Let $T=\inf \left\{t > \tau | X_t=X_1\right\}$. Then $T=\tau_{\kappa}$ where $\kappa$ is a stopping time of $F_{\tau_\kappa}$ and $\left\{\tau_{\kappa-1} \leq \tau\right\} \in F_{\tau_{\kappa-1}}$. By Wald's lemma,
\begin{eqnarray*}
E \displaystyle\sum_{t=1}^{T-1} I(X_t=x) = E \displaystyle\sum_{k=1}^{\kappa} N(x,B_k) = \pi_x El(B_1)E\kappa.
\end{eqnarray*}
\begin{eqnarray*}
E(T-1)=E\displaystyle\sum_{k=1}^{\kappa} l(B_k)=El(B_1)E\kappa.
\end{eqnarray*}

Again by irreducibility since the mean time to return to any state starting from $X_\tau$ is finite $E(T-\tau) \leq C_P$. Then for any $x \in S$,
\begin{eqnarray*}
&& N(x,T)-(T-\tau) \leq N(x,\tau)< N(x,T) \\
&& \pi_x E(T-1)-C_P\leq EN(x,\tau) \leq \pi_x E(T-1)+1 \\
&& \pi_x E\tau-C_P \leq E N(x,\tau) \leq \pi_x E(\tau) +C_P \\
&& |E N(x,\tau) - \pi_x E\tau| \leq C_P. \\
\end{eqnarray*}

\end{proof}

We are now ready to establish a relationship between the regret $R^\alpha(n)$ and $E^\alpha\left[T^{\alpha, i}(n)\right]$.

\begin{lemma}\label{lem:lemma2}
If the reward of each arm is given by a Markov chain satisfying the properties of Lemma \ref{lemma1},  then under any policy $\alpha$ for which the expected time between two successive samples from an arm is finite, we have
\begin{eqnarray}
R^\alpha(n) \leq  \displaystyle\sum_{i=1}^K (\mu^*-\mu^i)E^{\alpha}\left[T^{\alpha, i}(n)\right] + C_{\mathbf{S, P, r}} ~, 
\end{eqnarray}
where $C_{\mathbf{S, P, r}}$ is a constant that depends on all the state spaces $S^i$, transition probability matrices $P^i$, and the set of rewards $\mathbf{r}^i$, $i=1, \cdots, K$.
\end{lemma}

\begin{proof} 
Let $G^i = \vee_{j \neq i} F^j$. Since arms are independent $G^i$ is independent of $F^i$ where $F^i$ follows the definition in Lemma \ref{lemma1} applied to the $i$th arm.  Note that $T^{\alpha, i}(n)$ is a stopping time with respect to $\left\{G^i \vee F^i_n, n\geq 1\right\}$.  Let $X^i(1),...,X^i(T^{\alpha,i}(n))$ denote the successive states observed from arm $i$ up to $n$. Thus, $X^i(t)$ is the $t^{th}$ observation from arm $i$. Then the total reward obtained under policy $\alpha$ up to time $n$ is given by: 
\begin{eqnarray*}
\displaystyle  \sum_{t=1}^n r^{\alpha(t)}_{x_{\alpha(t)}} = \displaystyle\sum_{i=1}^K \sum_{j=1}^{T^{\alpha,i}(n)}\sum_{y\in S^i} r^i_y I(X^i(j)=y) ~. 
\end{eqnarray*}

By definition, the regret is $R^\alpha(n)=n\mu^*-E^\alpha\left[\displaystyle\sum_{t=1}^n r^{\alpha(t)}_{x_{\alpha(t)}} \right]$. Therefore 
\begin{eqnarray}
&&  \left|R^\alpha (n)-\left(n\mu^*- \sum_{i=1}^K \mu^iE^\alpha \left[T^{\alpha, i}(n)\right]\right)\right| \\
&=& \left|E^\alpha\left[\displaystyle\sum_{i=1}^K \sum_{j=1}^{T^{\alpha, i}(n)} \sum_{y\in S^i} r^i_y I(X^i(j)=y)\right] \right. \nonumber \\
&&  - \left. \displaystyle\sum_{i=1}^K\sum_{y \in S^i} r^i_y \pi^i_y E^\alpha \left[T^{\alpha, i}(n)\right] \right| 
\nonumber \\
&\leq&  \displaystyle\sum_{i=1}^K \sum_{y\in S^i} \left| E^\alpha \left[\sum_{t=1}^{T^{\alpha, i}(n)} r^i_y I(X^i(j)=y)\right] \right. \nonumber \\
&&  -  \left. r^i_y \pi^i_y E^\alpha \left[T^{\alpha, i}(n)\right]  \right|  \nonumber \\
&=& \displaystyle\sum_{i=1}^K \sum_{y \in S^i} r^i_y \left|E^\alpha \left[N(y,T^{\alpha, i}(n))\right]-\pi^i_y  E^\alpha \left[T^{\alpha, i}(n)\right]\right|  \nonumber \\
&\leq&  \displaystyle\sum_{i=1}^K\sum_{y \in S^i} r^i_y C_{P^i} = C_{\mathbf{S, P, r}} ~, 
\end{eqnarray}
where the first inequality follows from the triangle inequality and the fact that random variables corresponding to the states of the Markov chains are independent of the policy $\alpha$, the second equality follows since $T^{\alpha, i}(n)$ is a stopping time with respect to $\left\{G^i \vee F^i_n, n\geq 1 \right\}$, and the last inequality follows form Lemma \ref{lemma1}.
\end{proof}

Intuitively Lemma \ref{lem:lemma2} is quite easy to understand.  It states that the regret of any policy (its performance difference from always playing the best arm) is bounded by the sum over the expected differences from playing each non-optimal arms, subject to a constant. 

\section{An Index Policy and Its Regret Analysis} \label{sec:policy}

We consider the following sample-mean based index policy proposed by \cite{auer}, referred to as the UCB (upper confidence bound) policy. 

Denote by $r^i(k)$ the sample reward from arm $i$ when it is played for the $k$th time, 
and by $T^i(n)$ the total number of times arm $i$ has been played up to time $n$.  Then the sample mean reward from playing arm $i$ is given by $\bar{r}^i (T^i(n)) =\frac{r^i(1)+r^i(2)+...+r^i(T^i(n))}{T^i(n)}$.  
The policy defines an index for each arm, denoted by $g^i_{n,T^i(n)}$ for arm $i$, and plays at each time the arm with the highest index. 
%
%

The index is updated as follows.  Initially each arm is played exactly once.  For each arm played, its sample mean is updated; this corresponds to the first term of the index. If an arm is not played, then the uncertainty about the mean of the arm is updated; this corresponds to the second term of the index.  This algorithm is illustrated below. 
%

\begin{framed}
\noindent\textbf{UCB (Upper Confidence Bound)} \\
Initialization: $n=1$\\
for ($n \leq K$) \\
\indent play arm $n$; $n=n+1$. \\
while ($n>K$) \\ 
%
\indent $\bar{r}^i (T^i(n)) =\frac{r^i(1)+r^i(2)+...+r^i(T^i(n))}{T^i(n)}$; \\
\indent $g^i_{n,T^i(n)}= \bar{r}^i(T^i(n)) + \sqrt{\frac{L\ln n}{T^i(n)}},  \forall i$. \\
%
\indent Play the arm with the highest index. \\
\indent $n=n+1$. 
\end{framed}

In the index policy of \cite{auer} the constant $L$ is set to $2$. 
Below we will show that the regret of this policy grows at most logarithmically in $n$. To do this we will bound the expected number of plays of any non-optimal arm (with mean less than the mean of the best arm). We will use the following lemma by Gillman \cite{gillman}, which bounds the probability of a large deviation from the stationary distribution.

\begin{lemma}\label{lemma3} 
[Theorem 2.1 from \cite{gillman}] Consider a finite-state, irreducible, aperiodic and reversible Markov chain with state space $S$, matrix of transition probabilities $P$, and an initial distribution $\mathbf{q}$. Let $N_{\mathbf{q}}=\left\|(\frac{q_x}{\pi_x}, x\in S)\right\|_2$. Let $\epsilon=1-\lambda_2$, where $\lambda_2$ is the second largest eigenvalue of the matrix $P$. $\epsilon$ will be referred to as the eigenvalue gap. Let $A \subset S$. Let $t_A(n)$ be the number of times that states in the set $A$ are visited up to time $n$. Then for any $\gamma \geq 0$, we have
\begin{eqnarray}
P(t_A(n)-n \pi_A \geq \gamma) \leq (1+\frac{\gamma \epsilon}{10n})N_{\mathbf{q}}e^{-\gamma^2 \epsilon/20n},
\end{eqnarray}
where 
\begin{eqnarray*}
\pi_A=\displaystyle\sum_{x \in A} \pi_x.
\end{eqnarray*}
\end{lemma}

\begin{proof} 
See Theorem 2.1 of \cite{gillman}.
\end{proof}


We now state the main theorem of this section.  The proof follows similar methods as used in \cite{auer}. 

\begin{theorem}
\label{theorem1}
Assume all arms are modeled as finite state, irreducible, aperiodic, and reversible Markov chains, and assume all rewards are positive. 
Let $\pi_{\min} = \min_{x \in S^i, 1 \leq i \leq K} \pi^i_x$, $r_{\max}=\max_{1\leq i \leq K, x \in S^i} r^i_x$, $r_{\min}=\min_{1\leq i \leq K, x \in S^i} r^i_x$, $S_{\max} = \max_{1\leq i \leq K} |S^i|$, $\epsilon_{\max}= \max_{1\leq i \leq K} \epsilon^i$, $\epsilon_{\min}= \min_{1\leq i \leq K} \epsilon^i$, where $\epsilon^i$ is the eigenvalue gap of the $i$th arm.   Then using a constant $L \geq 90 S^2_{\max} r^2_{\max}/\epsilon_{\min}$, the regret of the UCB policy can be bounded above by
\begin{eqnarray}
R(n) &\leq&  4L\displaystyle\sum_{i: \mu^i < \mu^*} \frac{\ln n}{(\mu^*-\mu^i)} +  \displaystyle\sum_{i: \mu^i < \mu^*} (\mu^*-\mu^i) C^i \nonumber \\
&& +  C_{\mathbf{S,P,r}} ~, 
\end{eqnarray}
where 
\begin{eqnarray}
C^i &=& 1+(D^i+D^*)\beta, \nonumber \\
D^i &=& \frac{|S^i|}{\pi_{\min}} \left(1+ \frac{\epsilon_{\max}\sqrt{L}}{10 |S^i| r_{\min}}\right), \nonumber \\ \beta &=& \sum_{t=1}^{\infty} t^{-2}. \nonumber
\end{eqnarray}
\end{theorem}

\begin{proof}
Throughout the proof all quantities pertain to the UCB policy, which will be denoted by $\alpha$ and suppressed from the superscript whenever there is no ambiguity.  Let $\bar{r}^i(T^i(n))$ denote the sample mean of the reward collected from arm $i$ over the first $n$ plays. 
%
%
Let $c_{t,s}=\sqrt{L\ln t/s}$, and let $l$ be any positive integer. Then,
\begin{eqnarray}
&& T^i(n) = 1+\displaystyle\sum_{t=K+1}^{n} I(\alpha(t)=i) \nonumber \\
& \leq&  l+\displaystyle\sum_{t=K+1}^{n} I(\alpha(t)=i, T^i(t-1) \geq l)  \nonumber \\
& \leq&  l+\displaystyle\sum_{t=K+1}^{n} I\left(\gamma^i(t,l), T^i(t-1) \geq l\right) \nonumber \\
&\leq&  l+\sum_{t=K+1}^{n} I\left(\zeta^i(t,l)\right) \nonumber \\
&\leq&  l+ \displaystyle\sum_{t=1}^{\infty}\sum_{s=1}^{t-1}\sum_{s_i=l}^{t-1} I(\bar{r}^*(s)+c_{t,s} \leq \bar{r}^i(s_i)+c_{t,s_i}), \label{eqn1}
\end{eqnarray} 
where $\gamma^i(t,l)$ is the event that 
\begin{eqnarray*}
\bar{r}^*(T^*(t-1))+c_{t-1, T^*(t-1)} \leq \bar{r}^i(T^i(t-1)) + c_{t-1,T^i(t-1)},
\end{eqnarray*}
$\zeta^i(t,l)$ is the event that
\begin{eqnarray*}
\min_{0< s< t}(\bar{r}^*(s)+c_{t-1,s}) \leq \max_{l< s_i< t} (\bar{r}^i(s_i)+c_{t-1,s_i}).
\end{eqnarray*}

We now show that $\bar{r}^*(s)+c_{t,s} \leq \bar{r}^i(s_i)+c_{t,s_i}$ implies that at least one of the following holds.
\begin{eqnarray}
\bar{r}^*(s) &\leq& \mu^* - c_{t,s} \label{eqn2}\\ 
\bar{r}^i(s_i) &\geq& \mu^i+c_{t,s_i} \label{eqn3}\\
\mu^* &<& \mu^i+2c_{t,si} \label{eqn4}.
\end{eqnarray}
This is because if none of the above holds, then we must have 
\begin{eqnarray*}
\bar{r}^*(s)+c_{t,s} > \mu^* \geq \mu^i+2c_{t,si} > \bar{r}^i(s_i)+c_{t,s_i} ~, 
\end{eqnarray*}
which contradicts $\bar{r}^*(s)+c_{t,s} \leq \bar{r}^i(s_i)+c_{t,s_i}$. 

If we choose $s_i \geq 4L\ln n/(\mu^*-\mu^i)^2$, then
%
$2c_{t,si}\leq \mu^*- \mu^i$, 
which means (\ref{eqn4}) is false, and therefore at least one of (\ref{eqn2}) and (\ref{eqn3}) has to be true with this choice of $s_i$.  
%
We next take $l=\left\lceil \frac{4L\ln n}{(\mu^*-\mu^i)^2}\right\rceil$, and proceed from (\ref{eqn1}).  Taking expectation on both sides gives: 
\begin{eqnarray*}
E[T^i(n)] &\leq& \left\lceil \frac{4L\ln n}{(\mu^*-\mu^i)^2}\right\rceil \\
&+& \displaystyle\sum_{t=1}^{\infty}\sum_{s=1}^{t-1}\sum_{s_i=\left\lceil \frac{4L\ln n}{(\mu^*-\mu^i)^2}\right\rceil}^{t-1} P(\bar{r}^*(s)\leq\mu^*-c_{t,s}) \\
&+& \displaystyle\sum_{t=1}^{\infty}\sum_{s=1}^{t-1}\sum_{s_i=\left\lceil \frac{4L\ln n}{(\mu^*-\mu^i)^2}\right\rceil}^{t-1} P(\bar{r}^i(s_i) \geq \mu^i+c_{t,s_i}). 
\end{eqnarray*}

Consider an initial distribution ${\bf q}^i$ for the $i$th arm.  We have: 
\begin{eqnarray*}
N_{\mathbf{q}^i}=\left\|\left(\frac{q_{y}^i}{\pi_{y}^i},y\in S^i\right)\right\|_2 \leq \sum_{y \in S^i} \left\|\frac{q_{y}^i}{\pi_{y}^i}\right\|_2+ \leq \frac{1}{\pi_{\min}} , 
\end{eqnarray*}
where the first inequality follows from Minkowski inequality.
Let $n^i_y(t)$ denote the number of times state $y$ of arm $i$ is observed up to time $t$. Then,

\begin{eqnarray} 
&& P(\bar{r}^i(s_i) \geq \mu^i+c_{t,s_i}) \nonumber \\
&=& P\left( \sum_{y \in S^i} r^i_y n^i_y(s_i) \geq s_i  \sum _{y \in S^i} r^i_y \pi^i_y+s_i c_{t,s_i} \right) \nonumber \\
&=& P\left( \sum_{y \in S^i} (r^i_y n^i_y(s_i) -r^i_y s_i \pi^i_y ) \geq s_i c_{t,s_i} \right) \nonumber \\
&=& P\left( \sum_{y \in S^i} (-r^i_y n^i_y(s_i) + r^i_y s_i \pi^i_y ) \leq - s_i c_{t,s_i} \right) \label{correction1}
\end{eqnarray}
Consider a sample path $\omega$ and consider the events
\begin{eqnarray*}
A &=&\left\{\omega : \sum_{y \in S^i} (-r^i_y n^i_y(s_i)(\omega) + r^i_y s_i \pi^i_y ) \leq - s_i c_{t,s_i} \right\} \\
B&=& \bigcup_{y \in S^i} \left\{\omega : - r^i_y n^i_y(s_i)(\omega) + r^i_y s_i \pi^i_y \leq - \frac{s_i c_{t,s_i}} {|S^i|} \right\}
\end{eqnarray*}
If $\omega \notin B$ then,
\begin{eqnarray*}
- r^i_y n^i_y(s_i)(\omega) + r^i_y s_i \pi^i_y > - \frac{s_i c_{t,s_i}}{|S^i|}, \ \forall y \in S^i \\
\Rightarrow \sum_{y \in S^i} (-r^i_y n^i_y(s_i)(\omega) + r^i_y s_i \pi^i_y ) > - s_i c_{t,s_i}
\end{eqnarray*}
Thus $\omega \notin A$ so $P(A) \leq P(B)$. Then continuing from (\ref{correction1}) we have 
\begin{eqnarray}
&& P(\bar{r}^i(s_i) \geq \mu^i+c_{t,s_i}) \nonumber \\
&\leq& \sum_{y \in S^i} P\left( - r^i_y n^i_y(s_i) + r^i_y s_i \pi^i_y  \leq - \frac{s_ic_{t,s_i}}{|S^i|} \right) \\
&=& \sum_{y \in S^i} P\left( r^i_y n^i_y(s_i) - r^i_y s_i \pi^i_y  \geq \frac{s_ic_{t,s_i}}{|S^i|} \right) \\
&\leq& \sum_{y \in S^i} \left(1+ \frac{\epsilon^i\sqrt{L\ln t/ s_i}}{10 |S^i| r^i_y}\right)N_{\mathbf{q}^i}t^{-\frac{L\epsilon^i}{20(|S^i| r^i_y)^2}}  \label{eqn5} \\
&\leq& \sum_{y \in S^i} \left(1+ \frac{\epsilon_{\max}\sqrt{Lt}}{10 |S^i| r_{\min}}\right)N_{\mathbf{q}^i} t^{-\frac{L\epsilon_{\min}}{20 S^2_{\max} r^2_{\max}}} \nonumber \\ 
&\leq& \sum_{y \in S^i} \sqrt{t} \left(1+ \frac{\epsilon_{\max}\sqrt{L}}{10r_{\min}}\right)N_{\mathbf{q}^i} t^{-\frac{L\epsilon_{\min}}{20 S^2_{\max} r^2_{\max}}} \nonumber \\ 
&\leq& \frac{|S^i|}{\pi_{\min}} \left(1+ \frac{\epsilon_{\max}\sqrt{L}}{10 |S^i| r_{\min}}\right)t^{-\frac{L\epsilon_{\min}-10 S^2_{\max} r^2_{\max}}{20 S^2_{\max} r^2_{\max}}}  ~~\label{eqn8}
\end{eqnarray}
where (\ref{eqn5}) follows from Lemma \ref{lemma3}.  Similarly, we have 
\begin{eqnarray} 
&& P(\bar{r}^*(s)\leq\mu^*-c_{t,s}) \nonumber \\
&= & P(\displaystyle \sum_{y \in |S^*|} r^*_y (n^*_y(s)-s \pi^*_y) \leq -sc_{t,s}) \nonumber \\
&\leq& \sum_{y \in |S^*|} P( r^*_y n^*_y(s)- r^*_y s \pi^*_y \leq -sc_{t,s} ) \nonumber \\
&=& \sum_{y \in |S^*|} P((s-\displaystyle\sum_{x \neq y} n^*_x(s)) -  s(1- \displaystyle\sum_{x \neq y} \pi^*_x) \leq -\frac{sc_{t,s}}{r^*_y}) \nonumber \\
&=& \sum_{y \in |S^*|} P(r^*_y \displaystyle\sum_{x \neq y} n^*_x(s)- r^*_y s\displaystyle\sum_{x \neq y} \pi^*_x \geq sc_{t,s})  \nonumber \\
&\leq& \sum_{y \in |S^*|} \left(1+ \frac{\epsilon^*\sqrt{L\ln t/ s}}{10 |S^*| r^*_y}\right)N_{\mathbf{q}^*}t^{-\frac{L\epsilon^*}{20(|S^*| r^*_y)^2}} \label{eqn6} \\
&\leq& \frac{|S^*|}{\pi_{\min}} \left(1+ \frac{\epsilon_{\max}\sqrt{L}}{10 |S^*| r_{\min}}\right)t^{-\frac{L\epsilon_{\min}-10 S^2_{\max} r^2_{\max}}{20 S^2_{\max} r^2_{\max}}} \label{eqn7}  
\end{eqnarray}
where (\ref{eqn6}) again follows from Lemma \ref{lemma3}.
Then from (\ref{eqn8}) and (\ref{eqn7}), we have 
\begin{eqnarray}
&& E[T^i(n)] \leq \frac{4L\ln n}{(\mu^*-\mu^i)^2}+1 \nonumber \\
&+& (D^i+D^*) \sum_{t=1}^{\infty} \sum_{s=1}^{t-1} \sum_{s_i=1}^{t-1} t^{-\frac{L\epsilon_{\min}-10 S^2_{\max} r^2_{\max}}{20 S^2_{\max} r^2_{\max}}} \nonumber \\
&=& \frac{4L\ln n}{(\mu^*-\mu^i)^2}+1+ (D^i+D^*) \sum_{t=1}^{\infty} t^{-\frac{L\epsilon_{\min}-50 S^2_{\max} r^2_{\max}}{20 S^2_{\max} r^2_{\max}}}  \nonumber \\
&\leq& \frac{4L\ln n}{(\mu^*-\mu^i)^2}+1+ (D^i+D^*)\beta, \label{corrected2}
\end{eqnarray}
where
\begin{eqnarray}
D^i &=& \frac{|S^i|}{\pi_{\min}} \left(1+ \frac{\epsilon_{\max}\sqrt{L}}{10 |S^i| r_{\min}}\right), \nonumber \\
\beta &=& \sum_{t=1}^{\infty} t^{-2}, \nonumber
\end{eqnarray}
and the inequality in (\ref{corrected2}) follows from the assumption $L \geq 90 S^2_{\max} r^2_{\max}/\epsilon_{\min}$.
Thus we have obtained the following bound: 
\begin{eqnarray}
&&\displaystyle\sum_{i:\mu^i < \mu^*} (\mu^*-\mu^i) E[T^i(n)] \nonumber \\
&\leq& 4L\displaystyle\sum_{\mu^i < \mu^*} \frac{\ln n}{( \mu^*-\mu^i)} + \displaystyle\sum_{i:\mu^i<\mu^*} (\mu^*-\mu^i) C^i \label{eqn10}
\end{eqnarray}
where
\begin{eqnarray*}
C^i &=& 1+ (D^i + D^*)\beta.
\end{eqnarray*}
Using (\ref{eqn10}) in Lemma \ref{lem:lemma2} completes the proof. 
\end{proof}
%

\section{Discussion}\label{sec:discussion}

We have bounded the regret of the UCB policy uniformly over time by $\ln n$. While of the same order, this bound may be worse than the asymptotic bound given in \cite{anantharam1} in terms of the constant.  However, it holds uniformly over time and we have a very simple index based on the sample mean. The index policy in \cite{anantharam1} depends on all the previous sequence of observations and the calculation of the index requires integration over the parameter space and finding an infimum of a function over the set of parameters. 

The bound we have in Theorem \ref{theorem1} depends on the stationary distributions, eigenvalue gap from the arms and rewards. 
Note that the validity of this bound relies on selecting a sufficiently large value for the constant $L$ that requires the knowledge of the smallest eigenvalue gap.  
If we know a priori (or with high confidence) that there exist $c_1$, $c_2$ and $c_3$ such that $\epsilon_{\min}\geq c_1 > 0$, $0 < r_{\max} \leq c_2$ and $S_{\max} \leq c_3$ then setting $L=90 c_3^2 c_2^2/ c_1$ will be sufficient.  While this is a sufficient condition for the bound to hold and not necessary,  similar results under a weaker condition are not yet available.   

Selecting a large $L$ will increase the magnitude of the exploration component of the index which depends only on the current time and the total number of times the corresponding arm has been selected up to the current time. This means that the rate of exploration will increase, but the regret will remain logarithmic with time.  The only things that change are the constant and the multiplicative factor of the logarithmic term of the bound. 

In general the eigenvalue gap is a complex expression of the components of the stochastic matrix.  It can be simplified in special cases. In next section we give an example of the index policy. 

%
%


\section{An Example} \label{sec:app}

Consider a player who plays one of $K$ machines at each time.  Each machine can be in one of two states ``1'' and ``0'' and is modeled as an irreducible and aperiodic Markov chain. 
This requirement along with time reversibility is satisfied if $p^i_{00}>0, p^i_{11}>0$,  $i=1, 2, \cdots, K$. The stationary distribution of machine $i$ is 
\begin{eqnarray*}
\boldsymbol{\pi}^i=[\pi^i_0, \pi^i_1]=
\left[\frac{p_{10}^i}{p_{10}^i+p_{01}^i},\frac{p_{01}^i}{p_{10}^i+p_{01}^i}\right] ~, 
\end{eqnarray*}
and the eigenvalue gap is,
\begin{eqnarray*}
\epsilon^i=p_{10}^i+p_{01}^i ~. 
\end{eqnarray*}

Figures \ref{figure1} (S.1) and \ref{figure5} (S.2) show the simulation results for $5$ arms averaged over 100 runs with parameters given in table \ref{table1}. $90 S^2_{\max} r_{\max}^2/\epsilon_{\min}$ is $1458$ for S.1 and $1688.2$ for S.2.
\begin{table}
\begin{center}
    \begin{tabular}{ | l | l | l | l |  }
    \hline
    S.1  & $p_{01}$, $p_{10}$ & $r_0$, $r_1$ &$\pi_1$, $\mu$ \\ \hline
    ch.1 & .3, .5 & 1, 1.2 & .3750, 1.075 \\ \hline
    ch.2 & .2, .6 & 1, 1.7 & .2500, 1.175 \\ \hline
    ch.3 & .6, .3 & 1, 1.5 & .6667, 1.333 \\ \hline
    ch.4 & .7, .2 & 1, 1.8 & .7778, 1.622 \\ \hline
    ch.5 & .4, .8 & 1, 1.3 & .3333, 1.100 \\ \hline
    S.2  & $p_{01}$, $p_{10}$ & $r_0$, $r_1$ & $\pi_1$, $\mu$ \\ \hline
    ch.1 & .0001, .9975 & 1, 2 &.0001, 1.000  \\ \hline
    ch.2 & .0010, .9900 & 1, 2 &.0010, 1.001  \\ \hline
    ch.3 & .3430, .5100 & 1, 2 &.4021, 1.402  \\ \hline
    ch.4 & .1250, .7500 & 1, 2 &.1429, 1.143  \\ \hline
    ch.5 & .0270, .9100 & 1, 2 &.0288, 1.029  \\ \hline
    \end{tabular}
\end{center}
\caption{Parameters of the arms for S.1 and S.2}
\label{table1}
\end{table}

In figure \ref{figure1}, we see that the performance is better when $L=2$ (which violates our sufficient condition), compared to $L=2000 > 90 S^2_{\max} r_{\max}^2/\epsilon_{\min}$ which satisfies the condition of theorem \ref{theorem1}. The bound from theorem \ref{theorem1} in this case is $45150 \ln n + 62.8$.
As $\epsilon$ becomes smaller the Gillman bound becomes loser.  However, our results for the two-state arms suggest that even when $L$ is small (e.g., $L=2$) compared to $90 S^2_{\max} r_{\max}^2/\epsilon_{\min}$, the UCB policy works well (and indeed better as suggested by our numerical results); the resulting regret is at most logarithmic with $n$.  This seems to suggest that UCB's regret can be bounded logarithmically under any value of $L$ for two-state irreducible Markov chains.

We next compare the performance of UCB with the index policy given in \cite{anantharam1} assuming the player is restricted to playing one arm at a time. We generate the transition probability functions parametrized by $\theta \in [0,10]$, satisfying the conditions in \cite{anantharam1}. The parameter set is $\boldsymbol\theta=[0.5, 1, 7, 5, 3]$ where $i$th element corresponds to the parameter of arm $i$. Moreover, $\mu(\theta)$ is increasing in $\theta$, and $p_{01}(\theta)$ and $p_{10}(\theta)$ are log-concave in $\theta$ by letting 
\begin{eqnarray*}
p_{10}(\theta)=1-(\frac{\theta}{10})^2, \\
p_{01}(\theta)=(\frac{\theta}{10})^3.
\end{eqnarray*}
Any policy for which $R^\alpha(n)=o(n^\gamma)$ for every $\gamma>0$ is called a {\em uniformly good} policy by \cite{anantharam1}. It was shown that for any {\em uniformly good} policy $\alpha$,
\begin{eqnarray}
\liminf_{n \rightarrow \infty} \frac{R^\alpha(n)}{\ln n} \geq \sum_{j:\mu^j < \mu^*} \frac{\mu^* - \mu^j}{I(j,*)}, \label{bound1}
\end{eqnarray} 
where
\begin{eqnarray*}
I(j,*)=\sum_{x\in S} \pi^j_x \sum_{y \in S} p_{xy}(\theta^j) \ln \frac{p_{xy}(\theta^j)}{p_{xy}(\theta^*)} ~, 
\end{eqnarray*}
and that the index policy $\alpha^*$ in \cite{anantharam1} satisfies 
\begin{eqnarray*}
\limsup_{n \rightarrow \infty} \frac{R^{\alpha^*}(n)}{\ln n} \leq \sum_{j:\mu^j < \mu^*} \frac{\mu^* - \mu^j}{I(j,*)} ~. 
\end{eqnarray*}

Figure \ref{figure3} shows $p_{01}, p_{10},\mu$ we used in this set of experiments; Figure \ref{figure4} shows  $\ln p_{01}, \ln p_{10}$  as functions of $\theta$. Figure \ref{figure5} compares the regret of the index policy of \cite{anantharam1} (labeled as Anantharam's policy in the figure) with UCB under different values of $L$.
Note that the index policy of \cite{anantharam1} assumes the knowledge of $p_{01}(\theta)$ and $p_{10}(\theta)$, while in UCB these functions are unknown to the player. Simulation for $L= 1500 > 90 S^2_{\max} r_{\max}^2/\epsilon_{\min}$ satisfies the sufficient condition in theorem \ref{theorem1} for the bound to hold. The bound from theorem \ref{theorem1} for this case is $39846 \ln n + 45$, while Anantharam's bound is $4.406 \ln n$. 
             
The first thing to note is the gap between the bound we derived for UCB and the bound of \cite{anantharam1} given in (\ref{bound1}).  The second thing to note is that for $L=0.05$ UCB has smaller regret than the index policy of \cite{anantharam1}, as well as the bound in (\ref{bound1}), for the given time horizon.  Note that \cite{anantharam1} proved that the performance of any {\em uniformly good} policy cannot be better than the bound in (\ref{bound1}) asymptotically.  Since {\em uniformly good} policies have the minimum growth of regret among all policies,  this bound also holds for UCB.  This however is not a contradiction because this bound holds asymptotically; we indeed expect the regret of UCB with $L=0.05$ to be very close to this bound in the limit.  These results show that while the bound in \cite{anantharam1} is better than the bound we proved for UCB in this paper, in reality the UCB policy can perform very close to the tighter bound (uniformly, not just asymptotically).

\begin{figure}
\includegraphics[width=3.5in]{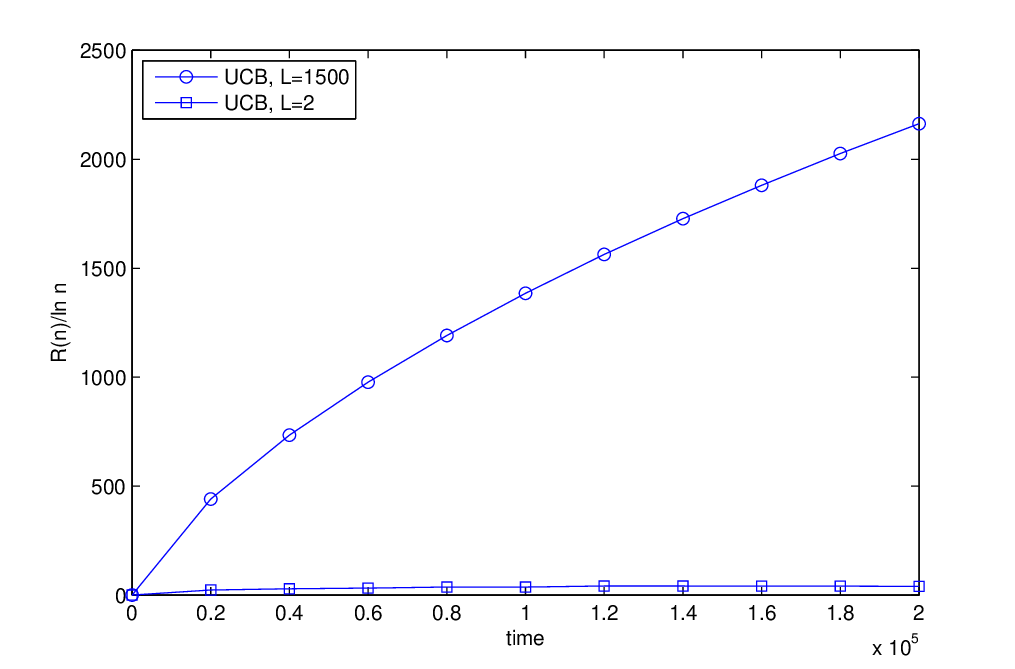}
\caption{Regret of UCB for S.1 }
\label{figure1}
\end{figure}

\begin{figure}
\includegraphics[width=3.5in]{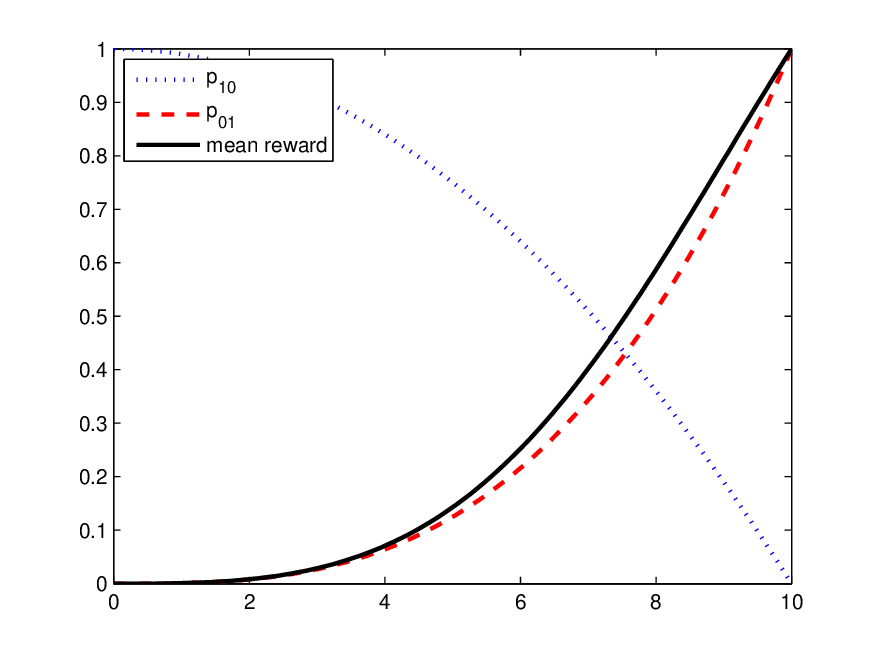}
\caption{$p_{01}, p_{10},\mu$ as functions of $\theta$}
\label{figure3}
\end{figure}

\begin{figure}
\includegraphics[width=3.5in]{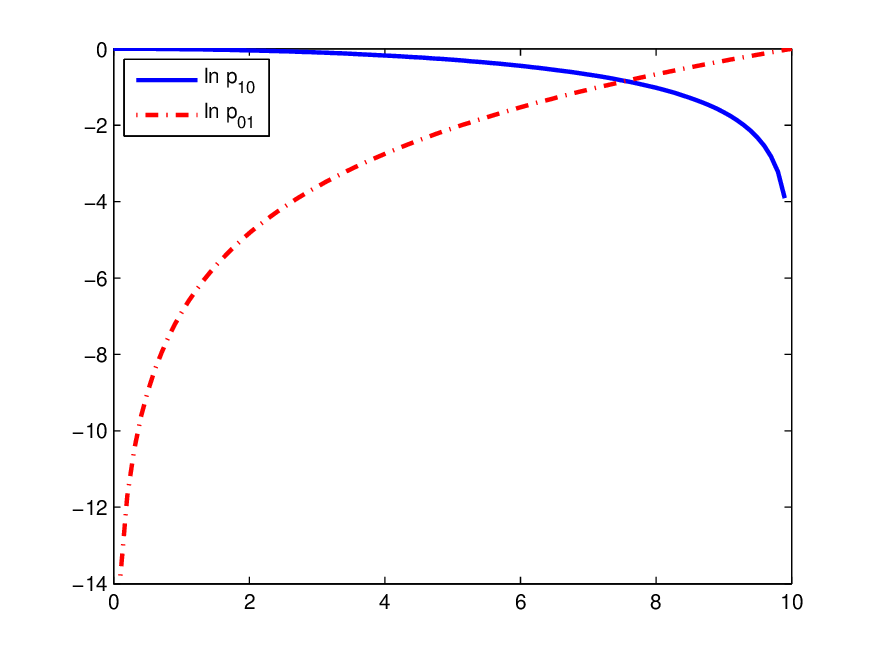}
\caption{$\ln p_{01}, \ln p_{10}$ as functions of $\theta$}
\label{figure4}
\end{figure}

\begin{figure}
\includegraphics[width=3.5in]{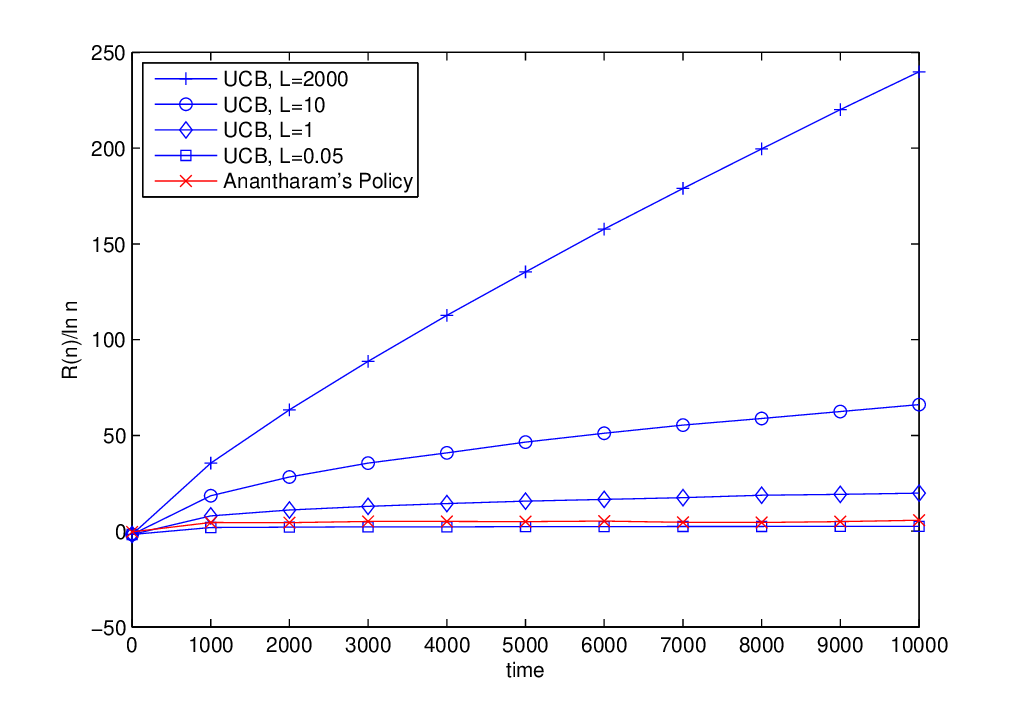}
\caption{Regrets of UCB and Anantharam's policy for S.2}
\label{figure5}
\end{figure}

\section{Conclusion} \label{sec:conc}

In this study we considered the multi-armed bandit problem with Markovian rewards, and proved that a sample mean based index policy achieves logarithmic regret uniformly over time provided that an exploration constant is sufficiently large with respect to the eigenvalue gaps of the stochastic matrices of the arms. 
An example was presented for a special case of two-state Markovian reward models.  Numerical results suggest that in this case order optimality of the index policy holds even when the sufficient condition on the exploration constant does not hold.

\end{document}